\documentclass[12pt]{article}
\usepackage{amsmath,amssymb,amsbsy,amsfonts,amsthm,latexsym,amsopn,amstext,
                            amsxtra,euscript,amscd}
\usepackage{booktabs}
\newtheorem{theorem}{Theorem}

\newtheorem{thm}[theorem]{Theorem}

\newtheorem{lem}[theorem]{Lemma}

\def\\{\cr}
\def\({\left(}
\def\){\right)}
\def\[{\left[}
\def\]{\right]}
\def\<{\langle}
\def\>{\rangle}

\def\cN{\mathcal N}

\def\notdivides{\mathrel{\kern-3pt\not\!\kern3.5pt\bigm|}}

\begin{document}


\title{Pairwise non-coprimality of triples}
\author{
{\sc Randell Heyman}\\
{School of Mathematics and Statistics, University of New South Wales} \\
{Sydney, NSW 2109, Australia}\\
{\tt randell@unsw.edu.au}}

\date{ }
\maketitle

\begin{abstract}
We say that $(a_1,\ldots,a_k)$ is \emph{pairwise non-coprime} if $\gcd(a_i,a_j) \ne 1$ for all $1 \le i <j \le k$. Let $a_1,a_2,a_3$ be positive integers less than $H$. We obtain an asymptotic formula for the number of $(a_1,a_2,a_3)$ that are pairwise non-coprime.
The probability that a randomly chosen unbounded positive integer triple is pairwise non-coprime is approximately 17.4\%. We also give an upper bound on the error term in an asymptotic formula for $\sum_{n=1}^H (\varphi(n)/n)^m$ for $m \ge 2$ and as $H \rightarrow \infty$.

\end{abstract}


\noindent
\section{Introduction}
The result regarding the probability that two positive integers are coprime is well-known (see, for example \cite[Theorem 332]{Har}).
Nymann \cite{Nym} gave the following result:
\begin{align}\label{Nyma}
\sum_{\substack{1\le a_1,\ldots,a_k\le H\\\gcd(a_1,\ldots, a_k)=1}}1&=\frac{H^k}{\zeta(k)}+\begin{cases} O(H \log H) & \text{if } k=2, \\
O\(H^{k-1}\) & \text{if } k\ge 3,
\end{cases}
\end{align}
where $\zeta(k)$ is the usual Riemann zeta function.
This naturally leads to enumeration of $k$-tuples with pairwise coprimality. T$\acute{\textrm{o}}$th \cite{Tot} showed that
\begin{align}\label{Toth}
\sum_{\substack{1\le a_1,\ldots,a_k\le H\\\gcd(a_i,a_j)=1\\i \ne j}}1&=\vartheta(k) H^k + O(H^{k-1}\(\log H \)^{k-1}),
\end{align}
where
$$\vartheta(k)=\prod_{p~\mathrm{prime}}\(1-\frac{1}{p}\)^{k-1}\(1+\frac{k-1}{p}\).$$
In this paper we enumerate the number of triples of maximum height $H$ that are pairwise non-coprime. Let
$$\cN_n(H)=\sum_{\substack{1\le a_1,\ldots,a_n\le H\\\gcd(a_i,a_j)\ne 1\\1\le i<j \le n}}1.$$
Our main result is the following.
\begin{thm}\label{main}
Suppose $H$ is a positive integer.
Then
$$\cN_3(H)=\rho H^3+O\(H^2\(\log H\)^2\),$$
where
$$\rho=1-\frac{3}{\zeta(2)}+3\prod_{p~\textrm{prime}}\(1 - \frac{2p-1}{p^3}\)- \prod_{p~\mathrm{prime}}\(1-\frac{3p-2}{p^3}\).$$
\end{thm}
According to Moree~\cite[Page 9]{Mor}, Freiberg~\cite{Fre} gives a result for the probability that three positive integers are pairwise non-coprime.
\section{Notation}
As usual, for any integer $n\ge 1,$ let $\omega(n)$, $\varphi(n)$ and $\tau(n)$ be the number of distinct prime factors, the
Euler totient function and the number of divisors of $n$ respectively (we also set $\omega(1) =0$).


We recall that the notation $f(x) = O(g(x))$  is
equivalent to the assertion that there exists a constant $c>0$ such that $|f(x)|\le c|g(x)|$ for all $x$. The notation
$f(x)=o(g(x))$ is equivalent to the assertion that $$\lim_{x \rightarrow \infty}\frac{f(x)}{g(x)}=0.$$

Finally, we use $(a,b)$ to represent $\gcd(a,b)$.

\section{Preparatory Lemma}
Kac \cite{Kac} attributes to I. Schur the following result. For $m \ge 2$,
$$\lim_{H \rightarrow  \infty}\frac{1}{H} \sum_{n=1}^H\(\frac{\varphi(n)}{n}\)^m = \prod_{p~\textrm{prime}}\(1+\frac{\(1-1/p\)^m-1}{p}\).$$
For Theorem \ref{main} we require an upper bound on the error term in an asymptotic formula for $$\sum_{n=1}^H \(\frac{\varphi(n)}{n}\)^2.$$
I thank an anonymous referee for pointing out that an upper bound on the error term in the general case is known (see \cite{Cho}) and that it has since been improved (see for e.g. ~\cite{Bal},~\cite{Liu}) using analytic tools. We provide a different elementary proof to that of ~\cite{Cho}.
\begin{lem}\label{variousm}
Let $m\ge 2$. We have
\begin{align*}
\sum_{n=1}^H \(\frac{\varphi(n)}{n}\)^m=H\prod_{p~\textrm{prime}}\(1+\frac{\(1-1/p\)^m-1}{p}\)+O\((\log H)^m\).
\end{align*}
\end{lem}
\begin{proof}
Let $m \ge 2$. Then
\begin{align}
\sum_{n=1}^H \(\frac{\varphi(n)}{n}\)^m&=\sum_{n=1}^H\prod_{p\mid n}\(1-1/p\)^m\notag\\
&=\sum_{n=1}^H\prod_{p\mid n}\(1+f(p)\),\label{beforef(n)}
\end{align}
where
$$f(n) = \begin{cases} \prod_{p\mid n}\(\(1-1/p\)^m-1\) & \text{if } n \text{ is square free,} \\
0 & \text{otherwise. }  \end{cases}$$
We will freely use the fact that
$$|f(n)|\le \prod_{p\mid n}\frac{m}{p}=\frac{m^{\omega(n)}}{n}.$$
Returning to \eqref{beforef(n)} we have
\begin{align}
\sum_{n=1}^H \(\frac{\varphi(n)}{n}\)^m&=\sum_{n=1}^H\sum_{d \mid n} f(d)\notag\\
&=\sum_{d \le H} f(d)\(\frac{H}{d}+O\(1\)\)\notag\\
&=H\sum_{d \le H} \frac{f(d)}{d} + O\(\sum_{d \le H} |f(d)|\).\label{sumvarphisquared}
\end{align}
For the error term in \eqref{sumvarphisquared} we note from \cite[III.3 Theorem 6]{Ten} that
\begin{align}\label{tenenbaum}
\sum_{l=1}^H m^{\omega(l)}&=O\(H(\log H)^{m-1}\).
\end{align}
Using \eqref{tenenbaum} and partial summation the error term in \eqref{sumvarphisquared} can given by
\begin{align}\label{phierror}
\sum_{d \le H} |f(d)|&=O\((\log H)^m\).
\end{align}

For the main term in \eqref{sumvarphisquared} we observe that
$$\sum_{d \le \infty}\frac{f(d)}{d}$$ is absolutely convergent since
\begin{align*}
\sum_{d >H}\Big|\frac{f(d)}{d}\Big|\le \sum_{d >H} \frac{m^{\omega(d)}}{d^2}\le \sum_{d >H} \frac{d^{o(1)}}{d^2}\le \sum_{d >H} \frac{1}{d^{2+o(1)}}=H^{-1+o(1)}.
\end{align*}
Thus
\begin{align}\label{fdond}
\sum_{d \le H} \frac{f(d)}{d}&=\sum_{d < \infty} \frac{f(d)}{d}-\sum_{d > H} \frac{f(d)}{d}\notag\\
&=\prod_{p~\textrm{prime}}\(1+\frac{f(p)}{p}\)+O\(H^{-1+o(1)}\)\notag\\
&=\prod_{p~\textrm{prime}}\(1+\frac{\(1-1/p\)^m-1}{p}\)+O\(H^{-1+o(1)}\).
\end{align}
Combining \eqref{sumvarphisquared},~\eqref{phierror} and \eqref{fdond}  completes the proof.
\end{proof}

\section{Proof of Theorem \ref{main}}
It is clear that
\begin{align*}
\cN_3(H)&=H^3-\sum_{\substack{1 \le a_1,a_2,a_3 \le  H\\(a_i,a_j)= 1\\\mathrm{for~some}~1\le i<j \le 3}}1.
\end{align*}
Then using the inclusion-exclusion principle we have
\begin{align*}
\sum_{\substack{1 \le a_1,a_2,a_3 \le  H\\(a_i,a_j)= 1\\\mathrm{for~some}~1 \le i<j \le 3}}1&=\sum_{\substack{1 \le a_1,a_2,a_3 \le  H\\1\le i<j\le 3 }}\sum_{(a_i,a_j)=1}1-\sum_{\substack{1 \le a_1,a_2,a_3 \le  H\\1 \le i<j<k \le 3 }}\sum_{\substack{(a_i,a_j)=1\\(a_j,a_k)=1}}1\\
&\quad\quad+\sum_{\substack{1 \le a_1,a_2,a_3 \le  H\\(a_1,a_2)=1\\ (a_1,a_3)=1\\(a_2,a_3)=1}}1.
\end{align*}
Using symmetry, we obtain
\begin{align}\label{firstpart}
\cN_3(H)&=H^3-3\sum_{\substack{1 \le a_1,a_2,a_3 \le  H\\(a_1,a_2)=1}}1+3\sum_{\substack{1 \le a_1,a_2,a_3 \le  H\\(a_1,a_2)=1\\(a_1,a_3)=1}}1-\sum_{\substack{1 \le a_1,a_2,a_3 \le  H\\(a_1,a_2)=1\\(a_1,a_3)=1\\(a_2,a_3)=1 }}1.
\end{align}
Using \eqref{Nyma}, the first summation of \eqref{firstpart} is given by
\begin{align}\label{firstsum}
\sum_{\substack{1 \le a_1,a_2,a_3 \le  H\\(a_1,a_2)=1}}1&=\frac{H^3}{\zeta(2)}+O\(H^2\log H\).
\end{align}
Using \eqref{Toth}, the third summation of \eqref{firstpart} is given by
\begin{align}\label{thirdsum}
\sum_{\substack{1 \le a_1,a_2,a_3 \le  H\\(a_1,a_2)=1\\(a_1,a_3)=1\\(a_2,a_3)=1 }}1&=\vartheta(3)H^3+O\(H^2(\log H)^2\).
\end{align}

It remains to express the middle summation of \eqref{firstpart}
as a multiple of $H^3$ and a suitable error term.
If we let
\begin{align}\label{phinh}
\varphi(n,H)&=\sum_{\substack{1\le a \le H\\(a,n)=1}}1,
\end{align}
then we have
\begin{align}\label{sumtophi}
\sum_{\substack{1 \le a_1,a_2,a_3 \le  H\\(a_1,a_2)=1\\ (a_1,a_3)=1}}1&=\sum_{\substack{1 \le n \le H\\}} \sum_{\substack{1\le a_2,a_3\le H\\(n,a_2)=1\\(n,a_3)=1}}1=\sum_{1 \le n \le H}\sum_{\substack{1\le a_3\le H\\(n,a_3)=1}}1\sum_{\substack{1\le a_2\le H\\(n,a_2)=1}}1\notag\\&=\sum_{1 \le n \le H}\varphi(n,H)^2.
\end{align}
The following is well-known (or see \cite[Lemma~4]{Hey}),
\begin{align*}
\varphi(n,H)&=\frac{H \varphi(n)}{n}+O\(2^{\omega(n)}\),
\end{align*}
and substituting into \eqref{sumtophi} we have,
\begin{align}\label{stage1}
\sum_{\substack{1 \le a_1,a_2,a_3 \le  H\\(a_1,a_2)=1\\(a_1,a_3)=1}}1&=\sum_{1 \le n \le H}\(\frac{H \varphi(n)}{n}+O\(2^{\omega(n)}\)\)^2\notag\\
\begin{split}
&=H^2\sum_{1 \le n \le H}\(\frac{ \varphi(n)}{n}\)^2+O\(H\sum_{1 \le n \le H}\frac{\varphi(n)2^{\omega(n)}}{n}\)\\
&\quad\quad\quad\quad+O\(\sum_{1 \le n \le H}\(2^{\omega(n)}\)^2\).
\end{split}
\end{align}
Appealing to \eqref{tenenbaum} we have
\begin{align}\label{stage2}
O\(\sum_{1 \le n \le H}\frac{\varphi(n)2^{\omega(n)}}{n}\)&=O\(\sum_{1 \le n \le H}2^{\omega(n)}\)=O\(H\log H\),
\end{align}
and also
\begin{align}\label{stage3}
\sum_{1 \le n \le H}\(2^{\omega(n)}\)^2&=O\(H\(\log H\)^3\).
\end{align}
Substituting equations \eqref{stage2} and \eqref{stage3} into \eqref{stage1} yields
\begin{align}\label{twogcd}
\sum_{\substack{1 \le a_1,a_2,a_3 \le  H\\(a_1,a_2)=1\\(a_1,a_3)=1}}1&=H^2\sum_{1 \le n \le H}\(\frac{ \varphi(n)}{n}\)^2+O\(H\(\log H\)^3\).
\end{align}

Using Lemma \ref{variousm}, setting $m=2$, and substituting into \eqref{twogcd} yields
\begin{align}\label{twogcdfinal}
\sum_{\substack{1 \le a_1,a_2,a_3 \le  H\\(a_1,a_2)=1\\(a_1,a_3)=1}}1&=H^3\prod_{p~\textrm{prime}}\(1 - \frac{2p-1}{p^3}\)+O\(H^2\(\log H\)^2\).
\end{align}
Substituting \eqref{firstsum}, \eqref{thirdsum} and \eqref{twogcdfinal} into \eqref{firstpart} completes the proof.
\section{Comments}
By using Theorem~\ref{main} we see that the probability that three randomly chosen positive integers will be pairwise non-coprime is given by $\rho \approx 0.1742$.

In this paper we have only considered $\cN_3(H)$. Our approach does not seem particularly well suited to higher tuples (that is $\cN_n(H)$ for $n>3$).
If we examine \eqref{firstpart} we observe that finding a suitable expression for $\cN_3(H)$ involved 3 different summations. The expression of two of these summations as a multiple of $H^3$ with a suitably bound error term was provided by previously known results. For $\cN_4(H)$ we have 10 summations (each summation corresponds to one of the, up to isomorphism, 10  non-null undirected graphs of 4 vertices). Of these 10 summations 6 can be obtained by natural extensions of the techniques in this paper. The remaining 4, namely
\begin{align*}
\sum_{\substack{1\le a_1,\ldots,a_4\le H\\(a_1,a_2)=1\\(a_2,a_3)=1\\(a_3,a_4)=1}}1,
\sum_{\substack{1\le a_1,\ldots,a_4\le H\\(a_1,a_2)=1\\(a_2,a_3)=1\\(a_3,a_4)=1\\(a_4,a_1)=1}}1,
\sum_{\substack{1\le a_1,\ldots,a_4\le H\\(a_1,a_2)=1\\(a_2,a_3)=1\\(a_2,a_4)=1\\(a_3,a_4)=1}}1\quad \text{and}~
\sum_{\substack{1\le a_1,\ldots,a_4\le H\\(a_1,a_2)=1\\(a_1,a_3)=1\\(a_1,a_4)=1\\(a_2,a_3)=1\\(a_2,a_4)=1}}1,
\end{align*} would appear to require different techniques.
\section{Acknowledgements}
The author would like to thank Igor Shparlinski for numerous helpful suggestions, particularly regarding two of the summations.

\end{document}